\documentclass[a4paper,11pt,reqno]{amsart}
\usepackage{amsmath,amssymb,amsthm} 
\usepackage{graphics} 

\usepackage{epsfig}
\usepackage{color}

\numberwithin{equation}{section}

\newtheorem{theorem}{Theorem}[section]
\newtheorem{lemma}[theorem]{Lemma}
\newtheorem{proposition}[theorem]{Proposition}

\theoremstyle{definition}
\newtheorem{definition}[theorem]{Definition}
\newtheorem{remark}[theorem]{Remark}

\newcommand{\restr}{\mathop{\raisebox{-.127ex}{\reflectbox{\rotatebox[origin=br]{-90}{$\lnot$}}}}}

\newcommand{\R}{\mathbb{R}}
\newcommand{\N}{\mathbb{N}}

\newcommand{\Z}{\mathbf{Z}}

\newcommand{\eps}{\varepsilon}

\newcommand{\be}{\begin{equation}}
\newcommand{\ee}{\end{equation}}

\newcommand\lt{\left}
\newcommand\rt{\right}

\def\les{\lesssim}
\def\ges{\gtrsim}

\def\EE{\mathbb{E}}
\def\PP{\mathbb{P}}
\def\diam{\operatorname{diam}}

\newcommand{\bra}[1]{\left( #1 \right)}
\newcommand{\sqa}[1]{\left[ #1 \right]}
\newcommand{\cur}[1]{\left\{ #1 \right\}}

\newcommand{\abs}[1]{\left| #1 \right|}

\def\fref{f^{\mathrm{ref}}}
\def\frefinf{f_{\infty}^{\mathrm{ref}}}
\def\fbi{f^{\textrm{bi}}}
\def\fbiinf{f_{\infty}^{\textrm{bi}}}
\usepackage{enumerate}

\title{
Convergence of asymptotic costs for random Euclidean matching problems}
\author[M. Goldman]{Michael Goldman}
\address{M.G.: Universit\'e de Paris, CNRS, Sorbonne-Universit\'e,  Laboratoire Jacques-Louis Lions (LJLL), F-75005 Paris, 
France}
\email{goldman@u-paris.fr}
\author[D. Trevisan]{Dario Trevisan}
\address{D.T.: Dipartimento di Matematica, Università degli Studi di Pisa, 56125 Pisa, Italy  }
\email{dario.trevisan@unipi.it}
\date{}

\begin{document}

\begin{abstract}
We investigate the average minimum cost of a bipartite matching between two samples of $n$ independent random points uniformly distributed on a unit cube in $d\ge 3$ dimensions,  where the matching cost between two points is given by any power $p \ge 1$ of their Euclidean 
distance. As $n$ grows, we prove convergence, after a suitable renormalization, towards a finite and positive constant.
We also consider the analogous problem of optimal transport between $n$ points and the uniform measure. The proofs combine sub-additivity inequalities with a PDE ansatz similar to the one proposed in the context of the matching problem in two dimensions and later extended to obtain upper bounds in higher dimensions.
\end{abstract}
\maketitle

\section{Introduction}
The aim of this paper is to extend the results of \cite{DoYu95,BoutMar,BaBo,DeScSc13} on the existence of the thermodynamic limit for some random optimal matching problems. Because of their relations to computer science, statistical physics and  quantization of measures, optimal matching problems 
have been the subject of intense research both from the mathematical and physical communities.
We refer for instance to \cite{Yu98,Ta14, CaLuPaSi14} for more details in particular regarding the vast literature.

Probably the simplest and most studied variant of these problems is the bipartite (or Euclidean bipartite) matching on the unit cube in $d$ dimensions. Given $p\ge 1$ and  two independent families of i.i.d.\ random variables $(X_i)_{i\ge 1}$ and $(Y_i)_{i \ge 1}$ with common law the uniform (Lebesgue) measure on $[0,1]^d$, the problem is to understand the behavior for large $n$ of 
\[
 \EE\lt[\frac{1}{n}\min_{\pi} \sum_{i=1}^n |X_i-Y_{\pi(i)}|^p\rt],
\]
where the minimum is taken  over all permutations $\pi$ of $\{1,\ldots, n\}$. It is by now well-known, see \cite{AKT84, BaBo, BoLe16, Le17} that for\footnote{The notation $A\ll 1$, which we only use in assumptions, means that there exists an $\eps>0$ only depending on the dimension $d$ and on $p\ge 1$, such that
if $A\leq \eps$ then the conclusion holds.
 Similarly, the notation $A\les B$, which we use in output statements, means that there exists a global
constant $C>0$ depending on the dimension $d$ and on $p\ge1$ such that $A\le C B$. We write $A\sim B$ if both $A\les B$ and $B\les A$.} 
 $n\gg1$ (see \cite{FoGu15} for some non-asymptotic bounds)
\begin{equation}\label{eq:boundE}
 \EE\lt[\frac{1}{n}\min_{\pi} \sum_{i=1}^n |X_i-Y_{\pi(i)}|^p\rt]\sim\begin{cases}
                                                                      n^{-\frac{p}{2}} &\textrm{ for } d=1\\
                                                                      \lt(\frac{\log n}{n}\rt)^{\frac{p}{2}} & \textrm{ for } d=2\\
                                                                      n^{-\frac{p}{d}} & \textrm{ for } d\ge 3.
                                                                     \end{cases}
\end{equation}
Let us point out that while the case $d\ge 3$, $p\ge d/2$ is not explicitly covered in the literature, the proof of \cite{Le17} clearly extends to any $p\neq 2$ (see also \cite{BobLe19}). Our main result is the following:
\begin{theorem}\label{theo:biintro}
 For every $d\ge 3$ and $p\ge 1$, there exists a  constant $\fbiinf=\fbiinf(p,d)>0$ such that 
 \begin{equation}\label{eq:theobi}
  \lim_{n\to \infty} n^{\frac{p}{d}}\EE\lt[\frac{1}{n}\min_{\pi} \sum_{i=1}^n |X_i-Y_{\pi(i)}|^p\rt]=\fbiinf.
 \end{equation}
\end{theorem}

This extends earlier results of \cite{DoYu95,BoutMar,BaBo,DeScSc13} where the same conclusion was obtained under the more restrictive condition $p<d/2$.
See also \cite{Ta92} for bounds on $ \fbiinf(1,d)$ as $d$ becomes large. As in the previously quoted papers, our proof is based on a sub-additivity argument and makes use of the classical observation that, thanks to the Birkhoff-Von Neumann Theorem, the bi-partite matching problem is actually an optimal transport problem. Indeed, if $\mu=\sum_{i=1}^n \delta_{X_i}$ and $\lambda=\sum_{i=1}^n \delta_{Y_i}$ 
are the associated empirical measures, then 
\[
 \min_{\pi} \sum_{i=1}^n |X_i-Y_{\pi(i)}|^p=W^p_p\lt(\mu,\lambda\rt),
\]
where $W_p$ denotes the Wasserstein distance  of order $p$ (see \cite{Viltop,Santam}). However, the papers \cite{BoutMar,BaBo,DeScSc13} rely then upon combinatorial arguments, and in fact their results apply to a larger class of random optimization problems, while \cite{DoYu95} strongly uses the dual formulation of optimal transport, which in the case $p=1$ is quite specific, since it becomes  a maximization over the set of $1$-Lipschitz functions.

The optimal transport point of view allows us to treat the defect in sub-additivity as a defect in local distribution of mass rather than a defect in local distribution of points. 
More precisely, even if $\mu([0,1]^d)=\lambda([0,1]^d)$ it is in general not true that for a given partition of $[0,1]^d$ in sub-cubes $Q_i$, $\mu(Q_i)=\lambda(Q_i)$. 
Therefore, in order to use sub-additivity, one needs to relax the definition of the  matching problem to take into account this disparity. In \cite{BoutMar,BaBo,DeScSc13} this is done by requiring that as many points as possible are matched.
Here we allow instead for varying  weights.  That is, for $\mu$ and $\lambda$ containing potentially a different number of points, we consider the problem
\[
 \EE\lt[W^p_p\lt(\frac{\mu}{\mu( [0,1]^d)},\frac{\lambda}{\lambda([0,1]^d)}\rt)\rt].
\]   
The main sub-additivity argument for this quantity is contained in Lemma \ref{lem:main}.
In order  to estimate the error in sub-additivity, we then use in Lemma \ref{lem:CZ} a PDE ansatz similar to the one proposed in the context of the matching problem in
\cite{CaLuPaSi14} and then used in \cite{AmStTr16} (see also \cite{Le17,GHO,BeCa}) to show that when $d=2$
\[
 \lim_{n\to \infty}\frac{n}{\log n} \EE\lt[\frac{1}{n}W_2^2(\mu,\lambda)\rt]=\frac{1}{2\pi}.
\]
Notice however that our use of this linearization ansatz is quite different from \cite{CaLuPaSi14,AmStTr16}. 
Indeed, for us the main contribution to the transportation cost is given by the transportation at the smaller scales and linearization is only used to estimate the higher order error term. 
On the other hand, in \cite{CaLuPaSi14,AmStTr16} (see also \cite{AmGlau}), the main contribution to the cost is given by the   linearized i.e. $H^{-1}$, cost. 
This is somewhat in line with the prediction by \cite{CaLuPaSi14} that for $d\ge 3$, the first order contribution to the Wasserstein cost is \textit{not} given by the linearized problem while higher order corrections are. 
In any case,  we give in  Proposition \ref{prop:grid} an alternative argument to estimate the error term without relying on the PDE ansatz. There we use instead an elementary comparison argument with 
one-dimensional transport plans. We included both proofs since we believe that they could prove useful in other contexts.

 Let us make a few more comments on the proof of Theorem \ref{theo:biintro}. As in \cite{DoYu95,BoutMar,BaBo},  the proof of \eqref{eq:theobi} is actually first done on a variant of the problem where the number of points follows a Poisson distribution instead of being deterministic. 
This is due to the fact that the restriction of a Poisson point process is again a Poisson point process. For this variant of the problem, rather than working on a fixed cube $[0,1]^d$
with an increasing density of points, we prefer to   make a blow-up at a scale $L=n^{1/d}$ and consider in Theorem \ref{theo:bi}, a fixed Poisson point process of intensity $1$ on $\R^d$ but restricted to  cubes $Q_L=[0,L]^d$ with $L\gg1$ (hence the terminology thermodynamic limit). 
We believe that the sub-additivity argument is slightly clearer in these variables (a similar rescaling is actually  implicitly used in \cite{BoutMar,BaBo}). This setting is somewhat reminiscent of \cite{HuSt13}, 
where super-additivity is used to construct an optimal coupling between the Poisson point process and the Lebesgue measure on $\R^d$. 
In order to pass from the Poisson version of the problem to the deterministic one, we prove a general \emph{de-Poissonization} result in Proposition \ref{prop:depois} which can hopefully be useful in other contexts. 

Besides the bipartite matching we also treat in Theorem \ref{theo:PoiLeb} and Theorem \ref{theo:match-ref-measure-unit-cube} the case of the matching to the reference measure. 
We actually treat this problem first since the proof is a bit simpler. Indeed, while the general scheme of the proof is identical to the bipartite case, the PDE ansatz 
used in Lemma \ref{lem:CZ} works well for ``regular'' measures and  a more delicate argument is required for the bipartite matching.  Notice that by Jensen and triangle inequalities, \eqref{eq:boundE} also holds for the matching to the reference measure. 

We  point out that in \cite{DoYu95, BaBo, DeScSc13}, it is more generally proven that if the points $X_i$ and $Y_i$ have a common law $\mu$ supported in $[0,1]^d$ instead of the Lebesgue measure 
(for measures with unbounded support a condition on the moments is required), then for $1\le p<d/2$ and $d\ge 3$,
\[
 \limsup_{n\to \infty} n^{\frac{p}{d}}\EE\lt[\frac{1}{n}\min_{\pi} \sum_{i=1}^n |X_i-Y_{\pi(i)}|^p\rt]\le \fbiinf\int_{[0,1]^d} \lt(\frac{d\mu}{dx}\rt)^{1-\frac{p}{d}}.
\]
 However, when $p>d/2$ and  without additional assumptions on $\mu$, the asymptotic rate may be different and thus  this inequality may fail, see e.g. \cite{FoGu15}. 
 Positive results for specific densities can be obtained nonetheless. For instance,  it is proven  in \cite{ledoux2019optimal} that for the standard Gaussian measure $\mu$ on $\R^d$, $d \ge 3$, the asymptotic bound
\[ n^{\frac{p}{d}} \EE\lt[\frac{1}{n}\min_{\pi} \sum_{i=1}^n |X_i-Y_{\pi(i)}|^p\rt]   \sim 1 \]
holds true also for $d/2 \le p < d$. 

Finally, we notice that usual results on concentration of measure allow us to improve from convergence of the expectations to strong convergence. However, we are able to cover only the case $1\le p < d$, see Remark~\ref{rem:strong-convergence}.

The plan of the paper is the following. In Section \ref{Sec:not}, we fix some notation and recall basic moment and concentration bounds for Poisson random variables. 
In Section \ref{Sec:main}, we state and prove our two main lemmas namely the sub-additivity estimate Lemma \ref{lem:main} and the error estimate Lemma \ref{lem:CZ}. 
In Section \ref{sec:matching-ref}, we then prove the existence of 
the thermodynamic limit for the matching problem of a Poisson point process to the reference measure. The analog result for the bipartite matching between two Poisson point processes 
is  obtained in Section \ref{sec:bipartite}. Finally, in Section \ref{Poitostan} we pass from the Poissonized problem to the deterministic one and discuss stronger convergence results.   

\section{Notation and preliminary results }\label{Sec:not}

We use the notation $|A|$ for the Lebesgue measure of a Borel set $A \subseteq \R^d$, and $\int_A f$ for the Lebesgue integral of a function $f$ on $A$. For $L>0$, we let $Q_L=[0,L]^d$. We denote by $|p|$ the  Euclidean norm of a vector $p\in \R^N$. For a function $\phi$, we use the notation $\nabla \phi$ for the gradient, $\nabla\cdot \phi$ for the divergence, $\Delta \phi$ for the Laplacian and $\nabla^2 \phi$ for the Hessian. 

\subsection{Optimal transport}

In this section we introduce some notation for the Wasserstein distance and recall few simple properties that will be used throughout. Proofs can be found in any of the monographs \cite{Viltop, Santam, peyre2019computational} with expositions of theory of optimal transport, from different perspectives.

Given $p \ge 1$, a Borel subset $\Omega \subseteq \R^d$ and two positive Borel measures $\mu$, $\lambda$ with $\mu(\Omega) = \lambda(\Omega) \in (0, \infty)$ and finite $p$-th moments, the Wasserstein distance of order $p\ge 1$ between $\mu$ and $\lambda$ is defined as the quantity
\[ W_p(\mu, \lambda) = \left(\min_{\pi\in\mathcal{C}(\mu,\lambda)} \int_{\R^d\times\R^d} {|x-y|}^p d \pi(x,y)\right)^\frac{1}{p},\]
where $\mathcal{C}(\mu, \lambda)$ is the set of couplings between $\mu$ and $\lambda$.
%
Moreover, if $\mu(\Omega) = \lambda(\Omega) = 0$, we define $W_p(\mu, \lambda) = 0$, while if $\mu(\Omega) \neq \lambda(\Omega)$, we let $W_p(\mu, \lambda) = \infty$.\\
Let us recall that since $W_p$ is a distance, we have the triangle inequality
\begin{equation}\label{eq:triangle}
W_p( \mu, \nu) \le W_p(\mu, \lambda) + W_p(\nu, \lambda).
\end{equation}
We will also  use the classical sub-additivity inequality 
\begin{equation}\label{eq:sub} W_p^p\bra{ \sum_{i} \mu_i, \sum_{i} \lambda_i} \le \sum_i W^p_p(\mu_i, \lambda_i),\end{equation}
for a finite set of positive measures $\mu_i$, $\lambda_i$.
This follows from the observation that if $\pi_i \in \mathcal{C}(\mu_i, \lambda_i)$, then $\sum_i \pi_i \in \mathcal{C}(\sum_i \mu_i, \sum_i\lambda_i)$.

\begin{remark}
In fact, our results deal with the \emph{transportation cost} $W_p^p(\mu, \lambda)$ rather than $W_p(\mu, \lambda)$. To keep notation simple, we write 
\[ W^p_{\Omega} (\mu, \lambda)  =  W_p^p (\mu\restr\Omega, \lambda\restr \Omega).\]
Moreover, if a measure is absolutely continuous with respect to Lebesgue measure, we only write its density. For example,
\[ W^p_{\Omega} \bra{ \mu,  \frac{\mu(\Omega)}{|\Omega|} },\]
denotes the transportation cost between $\mu \restr \Omega$ to the uniform measure on $\Omega$ with total mass $\mu(\Omega)$.

Occasionally we may write $W_{\Omega}(\mu, \lambda)$ instead of $\lt(W_{\Omega}^p(\mu, \lambda)\rt)^{1/p}$. This may lead to some ambiguity, but it should be clear from the context.
\end{remark}

\subsection{Poisson point processes}

As in \cite{DoYu95,BoutMar,BaBo}, we exploit invariance properties of Poisson point processes on $\R^d$ with uniform intensity in order to obtain simpler sub-additivity estimates.
We refer e.g.~to \cite{last2017lectures} for a general introduction to Poisson point processes. Here we only recall that a Poisson point process on $\R^d$ with intensity one can be defined as a random variable taking values on locally finite atomic measures
\[ \mu = \sum_{i} \delta_{X_i}\]
such that, for every $k \ge 1$, for any disjoint Borel sets $A_1,\ldots, A_k \subseteq \R^d$, the random variables $\mu(A_1),\ldots, \mu(A_k)$ are independent and $\mu(A_i)$ has a Poisson distribution of parameter $|A_i|$,
for every $i=1, \ldots, k$.
In particular, if $A\subseteq \R^d$ is Lebesgue negligible, then $\mu(A) = 0$ almost surely.

Existence of a Poisson point process of intensity one is obtained via a superposition argument, noticing that on every bounded subset $\Omega \subseteq \R^d$ the law of $\mu \restr \Omega$ can be  easily  described: conditionally on 
$\mu(\Omega) = n$, the measure $\mu \restr \Omega$ has the same law as the random measure
\[ \sum_{i=1}^n \delta_{X_i},\]
where $(X_i)_{i=1}^n$ are independent random variables with uniform law on $\Omega$. Uniqueness in law can be also obtained, so that translation invariance of Lebesgue measure entails that the process is stationary, i.e., any deterministic translation of the random measure $\mu$ leaves its law unchanged.

Let us finally recall the classical Cram\'er-Chernoff concentration bounds for Poisson random variables.
\begin{lemma}\label{lem:concentrationPoi}
 Let $N$ be a Poisson random variable with parameter $n\gg1$. Then, for every $ t>0$
 \begin{equation}\label{eq:Cramer}
  \PP[|N-n|\ge t]\le 2\exp\lt(-\frac{t^2}{2(t+n)}\rt).
 \end{equation}
As a consequence, for every $q\ge 1$,
\begin{equation}\label{eq:momentPoi}
 \EE\lt[|N-n|^q\rt]\les_q n^{\frac{q}{2}}.
\end{equation}

\end{lemma}
\begin{proof}
 For the concentration bound \eqref{eq:Cramer}, see for instance \cite{BouLuMa}. By the layer-cake representation, 
 \begin{align*}
  \EE\lt[|N-n|^q\rt]&\les \int_0^\infty t^{q-1}\exp\lt(-\frac{t^2}{2(t+n)}\rt) dt\\
  &\les \int_0^{\sqrt{n}} t^{q-1} dt+\int_{\sqrt{n}}^n t^{q-1} \exp\lt(-c \frac{t^2}{n}\rt) dt+\int_n^\infty t^{q-1}\exp(-ct) dt\\
  &\les_q n^{\frac{q}{2}}+  n^{\frac{q}{2}} +n^{q-1}\exp(-cn)\les_q n^{\frac{q}{2}}.
 \end{align*}
\end{proof}

\section{The main lemmas}\label{Sec:main}
Our sub-additivity argument rests on a general but relatively simple lemma (which we only apply here for rectangles).
\begin{lemma}\label{lem:main}
 For every $p\ge 1$, there exists a constant $C>0$ depending only on $p$ such that the following holds. For every Borel set  $\Omega\subset \R^d$, every Borel partition $(\Omega_i)_{i\in \N}$ of $\Omega$,
  every measures $\mu$, $\lambda$ on $\Omega$, and every $\eps\in(0,1)$, 
 \begin{equation}\label{eq:mainsub}
  W^p_{\Omega}\lt(\mu, \frac{\mu(\Omega)}{\lambda(\Omega)}\lambda\rt)\le (1+\eps)\sum_i W^p_{\Omega_i}\lt(\mu,\frac{\mu(\Omega_i)}{\lambda(\Omega_i)}\lambda\rt)+ \frac{C}{\eps^{p-1}}W^p_\Omega\lt(\sum_i \frac{\mu(\Omega_i)}{\lambda(\Omega_i)}\chi_{\Omega_i}\lambda,\frac{\mu(\Omega)}{\lambda(\Omega)}\lambda\rt).
 \end{equation}

\end{lemma}
\begin{proof}
 We first use the triangle inequality \eqref{eq:triangle} to get 
 \[
  W^p_{\Omega}\lt(\mu, \frac{\mu(\Omega)}{\lambda(\Omega)}\lambda\rt)\le \lt(W_\Omega\lt(\mu,\sum_i\frac{\mu(\Omega_i)}{\lambda(\Omega_i)}\chi_{\Omega_i}\lambda\rt)+  W_\Omega\lt(\sum_i \frac{\mu(\Omega_i)}{\lambda(\Omega_i)}\chi_{\Omega_i}\lambda,\frac{\mu(\Omega)}{\lambda(\Omega)}\lambda\rt)\rt)^p.
 \]
The proof is then concluded by combining the elementary inequality
\begin{equation}\label{eq:elementary}
 (a+b)^p\le (1+\eps) a^p+ \frac{C}{\eps^{p-1}} b^p \qquad \forall a,b>0 \textrm{ and } \eps\in (0,1),
\end{equation}
with the sub-additivity of $W_\Omega^p$  \eqref{eq:sub} in the form
\[
 W^p_\Omega\lt(\mu,\sum_i\frac{\mu(\Omega_i)}{\lambda(\Omega_i)}\chi_{\Omega_i}\lambda\rt)\le \sum_i W^p_{\Omega_i}\lt(\mu,\frac{\mu(\Omega_i)}{\lambda(\Omega_i)}\lambda\rt).
\]
\end{proof}
\begin{remark}
Alternatively, we could have also stated \eqref{eq:mainsub} in the slightly more symmetric form:
\begin{multline*}
  W^p_{\Omega}\lt(\frac{\mu}{\mu(\Omega)}, \frac{\lambda}{\lambda(\Omega)}\rt)\le (1+\eps)\sum_i \frac{\mu(\Omega_i)}{\mu(\Omega)}W^p_{\Omega_i}\lt(\frac{\mu}{\mu(\Omega_i)},\frac{\lambda}{\lambda(\Omega_i)}\rt)
 \\ + \frac{C}{\eps^{p-1}}W^p_\Omega\lt(\frac{1}{\mu(\Omega)}\sum_i \frac{\mu(\Omega_i)}{\lambda(\Omega_i)}\chi_{\Omega_i}\lambda,\frac{\lambda}{\lambda(\Omega)}\rt).
 \end{multline*}
However, the sub-additivity argument turns out to be a little bit simpler using  \eqref{eq:mainsub} instead. 
\end{remark}
Lemma \ref{lem:main} shows that in order to estimate the defect in sub-additivity, it is enough to bound the local defect of mass distribution. This will be done here through a  PDE argument. 
\begin{definition}\label{def:moderate}
 We say that a rectangle $R= x+\prod_{i=1}^{d}[0,L_i]$ is of moderate aspect ratio if for every $i,j$, $L_i/L_j\le 2$. A partition $\mathcal{R}=\{ R_i\}$  of $R$
 is called admissible if for every $i$, $R_i$ is  a rectangle of moderate aspect ratio and  $3^{-d}|R|\le |R_i|\le |R|$. Notice that in particular  $\#\mathcal{R}\les 1$ for every admissible partition.
\end{definition}

\begin{lemma}\label{lem:CZ}
 Let $R$ be a rectangle of moderate aspect ratio, $\mu$ and $\lambda$ be measures on $R$ with equal mass, both absolutely continuous with respect to Lebesgue and such that $\inf_R \lambda>0$. Then, for every $p\ge 1$
 \begin{equation}\label{eq:estimCZ}
  W_{R}^p(\mu,\lambda)\les \frac{\diam^p(R)}{(\inf_R \lambda)^{p-1}}\int_R |\mu-\lambda|^p.
 \end{equation}

\end{lemma}
\begin{proof}
 Let $\phi$ be a solution of the Poisson equation with  Neumann boundary conditions
 \begin{equation}\label{eq:Poi}
  \Delta \phi= \mu-\lambda \quad \textrm{in } R \qquad \textrm{ and } \qquad \nu\cdot \nabla \phi=0 \quad \textrm{on } \partial R.
 \end{equation}
We first argue that 
\begin{equation}\label{eq:Wpoi}
 W^p_{R}(\mu,\lambda)\les\frac{1}{(\inf_R \lambda)^{p-1}}\int_{R}|\nabla \phi|^p.
\end{equation}
Let us point out that this estimate is well-known 
and has already been used in the context of  the matching problem, see \cite{AmStTr16, GHO} in the case $p=2$ and \cite[Th. 2]{Le17} for general $p\ge 1$. 
Still, we give a proof for the reader's convenience. \\
We first argue as in \cite[Lem. 2.7]{GHO}, and use triangle inequality \eqref{eq:triangle} and the monotonicity of $W_R$  \eqref{eq:sub} to get 
\begin{align*}
 W_{R}(\mu,\lambda)&\le W_R\lt(\mu, \frac{1}{2}(\mu+\lambda)\rt)+ W_R\lt(\frac{1}{2}(\mu+\lambda),\lambda\rt)\\
 &\le W_R\lt(\frac{1}{2}\mu, \frac{1}{2}\lambda\rt)+ W_R\lt(\frac{1}{2}(\mu+\lambda),\lambda\rt)\\
 &=2^{-\frac{1}{p}}\lt( W_R(\mu,\lambda)+ W_R(\mu+\lambda,2\lambda)\rt)
\end{align*}
and thus 
\[
 W_{R}(\mu,\lambda)\les W_R(\mu+\lambda,2\lambda).
\]
We now recall that by the  Benamou-Brenier formula (see \cite[Th. 5.28]{Santam})
\[
  W_{R}^p(\mu+\lambda,2\lambda)=\min_{\rho,j}\lt\{ \int_0^1\int_{R} \frac{1}{\rho^{p-1}}|j|^p \ : \ \partial_t\rho+\nabla\cdot j=0, \rho_0=\mu+\lambda, \ \rho_1=2\lambda\right\}.
\]
Estimate \eqref{eq:Wpoi} follows using  
\[
 \rho_t= (1-t)\mu+t\lambda+\lambda \qquad \textrm{ and } \qquad j=\nabla \phi
\]
as competitor and noticing that for $t\in[0,1]$, $\rho_t\ge  \inf_R \lambda$.\\
We now claim that 
\begin{equation}\label{eq:claimCZ}
 \int_{R}|\nabla \phi|^p\les \diam^p(R)\int_R |\mu-\lambda|^p,
\end{equation}
which together with \eqref{eq:Wpoi} would conclude the proof of \eqref{eq:estimCZ}. Estimate \eqref{eq:claimCZ} is a direct consequence of Poincar\'e inequality and Calder\'on-Zygmund estimates for the Laplacian. 
However, since we did not find a precise reference for (global) Calder\'on-Zygmund estimates  on rectangles with Neumann boundary conditions, we give here a short proof.\\ 
  By scaling, we may assume that $\diam(R)=1$. Furthermore, using even reflections along $\partial R$ we may replace Neumann boundary conditions by periodic ones in \eqref{eq:Poi}. 
  By Poincar\'e inequality   \cite[Prop. 12.29]{leoni} (notice that thanks to the periodic boundary conditions we now have $\int_R \nabla \phi=0$),
  \[
   \int_{R}|\nabla \phi|^p\les \int_R|\nabla^2 \phi|^p. 
  \]
By interior Calder\'on-Zygmund estimates (see for instance \cite[Th. 7.3]{Giaquinta}),
 periodicity and the fact that $R$ has moderate aspect ratio, we get
\[
 \int_R |\nabla^2 \phi|^p\les \int_{R}|\mu-\lambda|^p +\lt(\int_{R} |\nabla^2 \phi|^2\rt)^{\frac{p}{2}}.
\]
By Bochner's formula and H\"older inequality,
\[
 \int_{R} |\nabla^2 \phi|^2=\int_R |\mu-\lambda|^2\les \lt(\int_R |\mu-\lambda|^p\rt)^{\frac{2}{p}},
\]
which concludes the proof of \eqref{eq:claimCZ}.
\end{proof}

\begin{remark}
 For $p=2$, combining the energy identity
 \[
  \int_R |\nabla \phi|^2=\int_{R}\phi(\lambda-\mu)
 \]
with Poincar\'e inequality, we see that  \eqref{eq:claimCZ} (and thus \eqref{eq:estimCZ}) holds for any convex set $R$. Although the situation for $p\ge 2$ is more subtle, see \cite[Prop. 2]{fromm},
we believe that \eqref{eq:claimCZ} holds for any rectangle, not necessarily of moderate aspect ratio. 
\end{remark}

\section{Matching to the reference measure}\label{sec:matching-ref}
In this section, we consider the optimal matching problem between $\mu$ a Poisson point process on $\R^d$ with intensity one and the Lebesgue measure. 
More precisely, for every $L\ge 1$ we let  
\[
 \fref(L)=\EE\lt[ \frac{1}{|Q_L|} W_{Q_L}^p(\mu,\kappa)\rt],
\]
where $Q_L=[0,L]^d$ and 
\[
 \kappa=\frac{\mu(Q_L)}{|Q_L|}\] is the generic constant for which this is well defined.  
 \begin{theorem}\label{theo:PoiLeb}
  For every $d\ge 3$ and  $p\ge 1$, the limit
  \[
   \frefinf=\lim_{L\to \infty} \fref(L)
  \]
exists and is strictly positive. Moreover, there exists $C>0$ depending on $p$ and $d$ such that for $L\ge 1$, 
\begin{equation}\label{eq:upboundPoiLeb}
 \frefinf\le \fref(L) +\frac{C}{L^{\frac{d-2}{2}}}.
\end{equation}

 \end{theorem}
 The proof follows the argument of \cite{BoutMar} (see also \cite{BaBo}) and is mostly based on the following sub-additivity estimate.
 \begin{proposition}\label{prop:subref}
  For every $d\ge 3$ and $p\ge 1$, there exists a constant $C>0$ such that for every $L\ge 1$ and $m\in \N$,
  \begin{equation}\label{subad}
 \fref(mL)\le \fref(L)+  \frac{C}{L^{\frac{d-2}{2}}}.
\end{equation}
 \end{proposition}
\begin{proof}
We start by pointing out that since $\fref(L)\les L^p$, it is not restrictive to assume that $L\gg1$ in the proof of \eqref{subad}.\\

 \textit{Step 1. }[The dyadic case] For the  sake of clarity, we start with the simpler case $m=2^k$ for some $k\ge 1$. We claim that 
 \begin{equation}\label{dyadic}
  \fref(2L)\le \fref(L)+ \frac{C}{L^{\frac{d-2}{2}}}.
 \end{equation}
 The desired estimate \eqref{subad} would then follow iterating \eqref{dyadic} and using that $\sum_{k\ge 0} \frac{1}{2^{k(d-2)}}<\infty$ for $d\ge 3$. In order to prove \eqref{dyadic}, we 
 divide the cube $Q_{2L}$ in $2^d$ sub-cubes $Q_i=x_i+ Q_L$ and let $\kappa_i=\frac{\mu(Q_i)}{|Q_L|}$ (and $\kappa=\frac{\mu(Q_{2L})}{|Q_{2L}|}$). Notice that we are considering a partition up to a Lebesgue negligible remainder, which gives no contribution almost surely. By \eqref{eq:mainsub}, for every $\eps\in(0,1)$,
 \[
  W^p_{Q_{2L}}(\mu,\kappa)\le (1+\eps)\sum_i W_{Q_i}^p(\mu,\kappa_i) +\frac{C}{\eps^{p-1}}W^p_{Q_{2L}}\lt(\sum_i \kappa_i \chi_{Q_i},\kappa\rt).
 \]
Dividing by $|Q_{2L}|$, taking expectations and using the fact that by translation invariance $\EE[\frac{1}{|Q_L|}W_{Q_i}^p(\mu,\kappa_i)]=\fref(L)$, we get 
\begin{align*}
 \fref(2L)&\le (1+\eps) \sum_i \frac{|Q_L|}{|Q_{2L}|}\fref(L)+\frac{C}{\eps^{p-1}}\EE\lt[\frac{1}{|Q_{2L}|}W^p_{Q_{2L}}\lt(\sum_i \kappa_i \chi_{Q_i},\kappa\rt)\rt]\\
 &=(1+\eps) \fref(L)+\frac{C}{\eps^{p-1}}\EE\lt[\frac{1}{|Q_{2L}|}W^p_{Q_{2L}}\lt(\sum_i \kappa_i \chi_{Q_i},\kappa\rt)\rt].
\end{align*}
We now estimate $\EE[\frac{1}{|Q_{2L}|}W^p_{Q_{2L}}(\sum_{i} \kappa_i \chi_{Q_i},\kappa)]$. Using
$\frac{1}{|Q_{2L}|}W^p_{Q_{2L}}(\sum_{i} \kappa_i \chi_{Q_i},\kappa)\les L^p \kappa$ together with 
\[
 \PP\lt[\kappa\le \frac{1}{2}\rt]\le \exp(-cL^d),
\]
which follows from the  Cram\'er-Chernoff bounds \eqref{eq:Cramer},  we may reduce ourselves to the event $\{\kappa\ge \frac{1}{2}\}$. 
Under this condition, by \eqref{eq:estimCZ}, we have 
\begin{align*}
 \frac{1}{|Q_{2L}|}W^p_{Q_{2L}}\lt(\sum_{i} \kappa_i \chi_{Q_i},\kappa\rt)&\les \frac{L^p}{L^d}\int_{Q_{2L}}\sum_i |\kappa_i-\kappa|^p \chi_{Q_i}\\
 &\les L^{p}\lt( |\kappa-1|^p+\sum_i|\kappa_i-1|^p\rt).
\end{align*}
Recalling that $\mu(Q_i)$ are Poisson random variables of parameter $|Q_i|$ and that $\kappa_i=\frac{\mu(Q_i)}{|Q_i|}$, we get from \eqref{eq:momentPoi}
\[
 \EE\lt[|\kappa-1|^p\rt]\sim\EE\lt[|\kappa_i-1|^p\rt]\les \frac{1}{L^{\frac{pd}{2}}}.
\]
Thus 
\begin{equation}\label{eq:claimsubref}
 \EE\lt[\frac{1}{|Q_{2L}|}W^p_{Q_{2L}}\lt(\sum_{i} \kappa_i \chi_{Q_i},\kappa\rt)\rt] \les \frac{1}{ L^{\frac{p}{2}(d-2)}}
\end{equation}
and  we conclude that for $\eps\in(0,1)$,
\begin{equation}\label{eq:dyadicnonopt}
 \fref(2L)\le (1+\eps)\fref(L) +\frac{C}{\eps^{p-1}}\frac{1}{ L^{\frac{p}{2}(d-2)}}.
\end{equation}
Optimizing in $\eps$ by choosing $\eps=L^{-(d-2)/2}$, and using that $\fref(L)$ is bounded 
(by \eqref{eq:boundE} and  \eqref{eq:Cramer}, see for instance \cite[Prop. 2.7]{GHO2} for details) we conclude the proof of \eqref{dyadic}.
Let us point out that we used here boundedness of $\fref(L)$ for simplicity but that as shown below it can also be obtained as a consequence of our proof.\\

\medskip
\textit{Step 2.}[The general case] We now consider the case when $m\in \N$ is not necessarily dyadic. 
We will partition $Q_{mL}$ into rectangles of almost dyadic size and thus need to deal with slightly more general configurations than dyadic cubes.
Let us introduce some notation.  Let $R$ be a rectangle with moderate aspect ratio and $\mathcal{R}=\{R_i\}$ be an admissible partition of $R$ (recall Definition \ref{def:moderate}).  
 Slightly abusing notation, we define 
\begin{equation}\label{def:frefR}
 \fref(R)= \EE\lt[\frac{1}{|R|}W^p_R(\mu,\kappa)\rt].
\end{equation}
\textit{Step 2.1.} We claim that the following variant of \eqref{eq:dyadicnonopt} holds: for every rectangle $R$ of moderate aspect ratio with $|R|\gg1$, every admissible partition $\mathcal{R}$ 
of $R$ and every $\eps\in(0,1)$, we have 
 \begin{equation}\label{onestep}
  \fref(R)\le (1+\eps)\sum_i \frac{|R_i|}{|R|} \fref(R_i) +\frac{C}{\eps^{p-1}} \frac{1}{|R|^{\frac{p(d-2)}{2d}}}.
 \end{equation}
Defining  $\kappa_i=\frac{\mu(R_i)}{|R_i|}$ and using \eqref{eq:mainsub} as above, we get
\[
  \fref(R)\le (1+\eps)\sum_i \frac{|R_i|}{|R|} \fref(R_i) +\frac{C}{\eps^{p-1}} \EE\lt[\frac{1}{|R|}W_{R}^p\lt(\sum_i \kappa_i \chi_{R_i},\kappa\rt)\rt].
\]
The estimate 
\begin{equation}\label{eq:claimsubrefR}
 \EE\lt[\frac{1}{|R|}W_{R}^p\lt(\sum_i \kappa_i \chi_{R_i},\kappa\rt)\rt]\les \frac{1}{|R|^{\frac{p(d-2)}{2d}}}
\end{equation}
is then obtained arguing exactly as for \eqref{eq:claimsubref}, using first the Cr\'amer-Chernoff bound \eqref{eq:Cramer} to reduce to the event $\{\kappa\ge \frac{1}{2}\}$
and then \eqref{eq:estimCZ} (recalling that $\diam(R)\sim|R|^{\frac{1}{d}}$ since $R$ has moderate aspect ratio) in combination with \eqref{eq:momentPoi} and the fact that
$\#\mathcal{R}\les 1$ since $\mathcal{R}$ is an admissible partition. \\

\smallskip
\textit{Step 2.2.} Starting from the cube $Q_{mL}$, let us construct a sequence of finer and finer partitions of $Q_{mL}$ by rectangles of moderate aspect ratios and 
side-length given by integer multiples of $L$.
 We let $\mathcal{R}_0=\{Q_{mL}\}$ and define $\mathcal{R}_k$ inductively as follows.
 Let    $R\in \mathcal{R}_k$. Up to translation we may assume that $R=\prod_{i=1}^d (0, m_i L)$ for some $m_i\in \N$. We then split each interval $(0,m_i L)$ into $(0,\lfloor\frac{m_i}{2}\rfloor L)\cup(\lfloor\frac{m_i}{2}\rfloor L, m_i L)$. 
 It is readily seen that this induces an admissible partition of $R$. Let us point out that when $m_i=1$ for some $i$, the corresponding interval   $(0,\lfloor\frac{m_i}{2}\rfloor L)$ is empty.
 This procedure stops after a finite number of steps $K$ once $\mathcal{R}_K=\{Q_L+z_i, z_i\in [0,m-1]^d\}$. It is also readily seen that $2^{K-1}<m\le 2^K$ and that for every $k\in [0,K]$ and every $R\in \mathcal{R}_k$ we have $|R|\sim (2^{K-k} L)^d$. \\
Let us prove by  a downward induction that there exists $\Lambda>0$ such that  for every $k\in [0,K]$ and every $R\in \mathcal{R}_{k}$,
\begin{equation}\label{induction}
 \fref(R)\le \fref(Q_L)+ \Lambda(1+\fref(Q_L))L^{-\frac{d-2}{2}} \sum_{j=K-k}^K 2^{- j\frac{d-2}{2}}.
\end{equation}
This is clearly true for $k=K$. Assume that it holds true for $k+1$. Let $R\in \mathcal{R}_{k}$. Applying \eqref{onestep} with $\eps= (2^{K-k} L)^{-(d-2)/2}\ll1$, we get 
\begin{align*}
 \fref(R)&\le (1+ \eps) \sum_{R_i\in \mathcal{R}_{k+1}, R_i\subset R} \frac{|R_i|}{|R|} \fref(R_i) + \frac{C}{\eps} \frac{1}{|R|^{\frac{p(d-2)}{2d}}}\\
 &\stackrel{\eqref{induction}}{\le} (1+\eps) \lt(\fref(Q_L)+ \Lambda(1+\fref(Q_L))L^{-\frac{d-2}{2}} \sum_{j=K-k+1}^K 2^{- j\frac{d-2}{2}}\rt) +  C (2^{K-k} L)^{-\frac{d-2}{2}}\\
 &\le  \fref(Q_L)+  \Lambda(1+\fref(Q_L))L^{-\frac{d-2}{2}}\\
 &\times\lt[\sum_{j=K-k+1}^K 2^{- j\frac{d-2}{2}}+2^{-(K-k)\frac{d-2}{2}}\lt( \frac{C}{\Lambda}+L^{-\frac{d-2}{2}} \sum_{j=K-k+1}^K 2^{- j\frac{d-2}{2}}  \rt)\rt].
\end{align*}
If $L$ is large enough, then $(\sum_{j=K-k+1}^K 2^{- j\frac{d-2}{2}}) L^{-(d-2)/2}\le \frac{1}{2}$. Finally, choosing $\Lambda\ge 2C$ yields \eqref{induction}.\\
Applying \eqref{induction} to $R=Q_{mL}$ and using that $\sum_{j\ge 0} 2^{- j\frac{d-2}{2}}<\infty$, we get 
\[
 \fref(mL)\le \fref(L)+ C(1+\fref(L)) \frac{1}{L^{\frac{d-2}{2}}}.
\]
Since $\fref(L)\les L^p$, writing that every $L\gg1$ may be written as $L=mL'$ for some $m\in \N$ and $L'\in[1,2]$, we conclude that $\fref(L)$ is bounded and thus \eqref{subad} follows.
\end{proof}
\begin{remark}
 We point out that as a consequence of the proof of Proposition \ref{prop:subref} we have for every rectangle $R$ of moderate aspect ratio (recall definition \eqref{def:frefR})
 \begin{equation}\label{eq:boundfR}
  \fref(R)\les 1.
 \end{equation}

\end{remark}

We can now prove Theorem \ref{theo:PoiLeb}
 \begin{proof}[Proof of Theorem \ref{theo:PoiLeb}]
  The existence of a limit $\frefinf$ is obtained from \eqref{subad}
  arguing exactly as in \cite{BoutMar} using the continuity of $L\mapsto \fref(L)$ (which can be obtained for instance by dominated convergence). 
  The fact that $\fref(L)\ges 1$ and thus $\frefinf>0$ follows from \eqref{eq:boundE} and \eqref{eq:Cramer}. Finally, \eqref{eq:upboundPoiLeb} follows from \eqref{subad} by sending $m\to \infty$ for fixed $L$.
 
 \end{proof}
 \begin{remark}
  It may be conjectured from \cite{CaLuPaSi14} that 
  \[
   |\fref(L)-\frefinf|\les \frac{1}{L^{d-2}}.
  \]
Let us notice that if we could replace  the term $\eps^{-(p-1)}$ in \eqref{eq:dyadicnonopt} by a constant then letting $\eps\to 0$   we would get the lower bound
\[
 \fref(L)-\frefinf\ges -\frac{1}{L^{\frac{p(d-2)}{2}}},
\]
which, at least for $p=2$, is in line with the conjectured rate. See also Remark~\ref{rem:rates} for   rates in the case of a deterministic number of points.
 \end{remark}

%
%

\section{Bi-partite matching}\label{sec:bipartite}
We now turn to the bi-partite matching. For $\mu$ and $\lambda$ two independent Poisson point processes of intensity one, we want to study the asymptotic behavior as $L\to \infty$ of 
\[
 \fbi(L)=\EE\lt[\frac{1}{|Q_L|}W_{Q_L}^p(\mu, \kappa\lambda)\rt],
\]
with $\kappa=\frac{\mu(Q_L)}{\lambda(Q_L)}$. Analogously to Theorem \ref{theo:PoiLeb}, we have:
\begin{theorem}\label{theo:bi}
  For every $d\ge 3$ and  $p\ge 1$, the limit
  \[
   \fbiinf=\lim_{L\to \infty} \fbi(L)
  \]
exists and is strictly positive. Moreover, there exists $C>0$ depending on $p$ and $d$ such that for $L\ge 1$, 
\begin{equation*}\label{eq:upboundbi}
 \fbiinf\le \fbi(L) +\frac{C}{L^{\frac{d-2}{2p}}}.
\end{equation*}
\end{theorem}
The proof of Theorem \ref{theo:bi} follows the same line of arguments as for Theorem \ref{theo:PoiLeb}. We only detail the estimate of the sub-additivity defect i.e. the counterpart of \eqref{eq:claimsubrefR}, since it is more delicate in the bipartite case. 
\begin{proposition}\label{prop:subbi}
 Let $R$ be a rectangle of moderate aspect ratio with $|R|\gg1$ and $\mathcal{R}=\{R_i\}$ be an admissible partition of $R$ (recall Definition \ref{def:moderate}).
Defining $\kappa_i= \frac{\mu(R_i)}{\lambda(R_i)}$ and $\kappa=\frac{\mu(R)}{\lambda(R)}$, we have for every $d\ge 3$ and $p\ge 1$,
\begin{equation}\label{eq:claimsubbi}
 \EE\lt[\frac{1}{|R|} W_{{R}}^p\lt(\sum_i \kappa_i \chi_{R_i} \lambda, \kappa \lambda\rt)\rt]\les \frac{1}{|R|^{\frac{d-2}{2d}}}.
\end{equation}
\end{proposition}
\begin{proof}
As opposed to \eqref{eq:claimsubrefR}, since $\lambda$ is atomic, we cannot directly use \eqref{eq:estimCZ}. The idea is to use as intermediate step the matching between $\lambda$ and the reference measure. 
Let us point out that 
since  $\fref(R_i)=\EE\lt[\frac{1}{|R_i|}W^p_{R_i}\lt(\lambda, \frac{\lambda(R_i)}{|R_i|}\rt)\rt]$ is of order one, we cannot apply naively the triangle inequality. 
The main observation is that since $|\kappa_i-\kappa|\ll1$ with overwhelming probability, the  amount of mass which actually needs to be transported is very small.\\

 Let $1\gg\theta>0$ to be optimized later on. 
For every $i$ let 
\[\theta_i:= (\kappa-\kappa_i)+\theta\]
and assume that 
\begin{equation}\label{assumelambda}
 \theta\ge 2\max_i |\kappa-\kappa_i|
\end{equation}
so that $\frac{3}{2}\theta\ge \theta_i\ge \frac{1}{2}\theta>0$. Notice that thanks to the Cr\'amer-Chernoff bounds \eqref{eq:Cramer}, 
\eqref{assumelambda} is satisfied with overwhelming probability as long as $\theta\gg |R|^{-1/2}$. Using the triangle inequality \eqref{eq:triangle} we have
\begin{align*}
 \lefteqn{W_{R}^p\lt(\sum_i \kappa_i \chi_{R_i} \lambda, \kappa \lambda\rt)}\\
 &\les W_{R}^p\lt(\sum_i \kappa_i \chi_{R_i} \lambda, \sum_i  \chi_{R_i} \lt[ (\kappa_i-\theta) \lambda + \theta \frac{\lambda(R_i)}{|R_i|} dx\rt]\rt)\\
 &+ W_{R}^p\lt(\sum_i  \chi_{R_i} \lt[ (\kappa_i-\theta) \lambda + \theta \frac{\lambda(R_i)}{|R_i|} dx\rt],\sum_i  \chi_{R_i} \lt[ (\kappa_i-\theta) \lambda + \theta_i \frac{\lambda(R_i)}{|R_i|} dx\rt]\rt)\\
 &+  W_{R}^p\lt(\sum_i  \chi_{R_i} \lt[ (\kappa_i-\theta) \lambda + \theta_i \frac{\lambda(R_i)}{|R_i|} dx\rt], \kappa \lambda\rt).
\end{align*}
Notice that $\sum_i\theta \lambda(R_i)=\sum_i\theta_i\lambda(R_i)$ by definition of $\theta_i$ and the fact that $\kappa\lambda(R)=\mu(R)=\sum_i \kappa_i \lambda(R)$ so that the second term is well defined.\\
We now estimate the three terms separately. The first and third are estimated in a similar way and we thus focus only on the first one. By sub-additivity \eqref{eq:sub} of $W^p_{R}$, we have 
\[W_{R}^p\lt(\sum_i \kappa_i \chi_{R_i} \lambda, \sum_i  \chi_{R_i} \lt[ (\kappa_i-\theta) \lambda + \theta \frac{\lambda(R_i)}{|R_i|} dx\rt]\rt)
 \le \theta \sum_i W_{R_i}^p\lt(\lambda, \frac{\lambda(R_i)}{|R_i|} \rt).
\]
We turn to the middle term. Again by sub-additivity of $W^p_{R}$,
\begin{multline*}
W_{R}^p\lt(\sum_i  \chi_{R_i} \lt[ (\kappa_i-\theta) \lambda + \theta \frac{\lambda(R_i)}{|R_i|} dx\rt],\sum_i  \chi_{R_i} \lt[ (\kappa_i-\theta) \lambda + \theta_i \frac{\lambda(R_i)}{|R_i|} dx\rt]\rt) \\
 \le W_{R}^p\lt(\sum_i \chi_{R_i} \theta \frac{\lambda(R_i)}{|R_i|} , \sum_i \chi_{R_i} \theta_i \frac{\lambda(R_i)}{|R_i|} \rt).
\end{multline*}
Using  \eqref{eq:estimCZ} in the event  $\{\lambda(R_i)/|R_i|\sim 1\}$ (which has overwhelming probability) and recalling that we assumed
 $\frac{3}{2}\theta\ge \theta_i\ge \frac{1}{2}\theta$, we have
\[
  W_{R}^p\lt(\sum_i \chi_{R_i} \theta \frac{\lambda(R_i)}{|R_i|} , \sum_i \chi_{R_i} \theta_i \frac{\lambda(R_i)}{|R_i|} \rt)\les\frac{\diam^p(R)}{\theta^{p-1}}\sum_i \lt|\kappa-\kappa_i\rt|^p |R_i|.
\]
Putting these two estimates together, dividing by $|R|$ and taking expectations we find
\begin{align*}
\EE\lt[\frac{1}{|R|} W_{R}^p\lt(\sum_i \kappa_i \chi_{R_i} \lambda, \kappa \lambda\rt)\rt]&\les \theta \sum_i\fref(R_i) + \frac{\diam^p(R)}{\theta^{p-1}} \sum_i \EE\lt[\lt|\kappa-\kappa_i\rt|^p\rt]\\
&\stackrel{\eqref{eq:boundfR}\&\eqref{eq:momentPoi}}{\les} \theta+ \frac{1}{\theta^{p-1}}\frac{1}{ |R|^{p\frac{(d-2)}{2d}}}.
\end{align*}
Optimizing in $\theta$ by choosing $\theta= |R|^{-\frac{d-2}{2d}}\gg |R|^{-1/2}$ (so that \eqref{assumelambda} is satisfied) this yields \eqref{eq:claimsubbi}.\\
\end{proof}

Comparing \eqref{eq:claimsubref} and \eqref{eq:claimsubbi}, one may wonder if \eqref{eq:claimsubbi} is suboptimal and could be improved. Let us prove that it is not the case, at least if we consider the slightly more regular situation of the matching to a deterministic grid.
\begin{proposition}\label{prop:grid}
 Let $\lambda=\sum_{X\in\Z^d} \delta_{X}$ and for every $L\gg1$, $(\mu_1,\cdots, \mu_{2^d})$ be $2^d$ independent Poisson random variables of parameter $L^d$. Writing for $L\gg1$, 
 $Q_{2L}=\cup_{i=1}^{2^d} Q_i$ where $Q_i$ are disjoint cubes of sidelength $L$ and defining $\kappa_i=\frac{\mu_i}{\lambda(Q_i)}$, $\kappa=\frac{\sum_i \mu_i}{\lambda(Q_{2L})}$, we have for $d\ge 3$ and $p\ge 1$,
\begin{equation*}\label{eq:grid}
 \EE\lt[\frac{1}{|Q_{2L}|} W_{Q_{2L}}^p\lt(\sum_i \kappa_i \chi_{Q_i} \lambda, \kappa \lambda\rt)\rt]\sim \frac{1}{L^{\frac{d-2}{2}}}.
\end{equation*}

 \end{proposition}
\begin{proof}
 \textit{Step 1.}[The lower bound] We start by proving that 
 \begin{equation}\label{eq:claimlowergrid}
  \EE\lt[\frac{1}{|Q_{2L}|} W_{Q_{2L}}^p\lt(\sum_i \kappa_i \chi_{Q_i} \lambda, \kappa \lambda\rt)\rt]\ges \frac{1}{L^{\frac{d-2}{2}}}.
 \end{equation}
For this we notice that if  $\pi$ is any admissible coupling, then $|x-y|\ge 1$ for every   $(x,y)$ in the support of $\pi$ with $x\neq y$ and thus\footnote{recall that $ W^1_{Q_{2L}}$ denotes the $1-$Wasserstein distance.} 
 \[
  W_{Q_{2L}}^p\lt(\sum_i \kappa_i \chi_{Q_i} \lambda, \kappa \lambda\rt)\ge W^1_{Q_{2L}}\lt(\sum_i \kappa_i \chi_{Q_i} \lambda, \kappa \lambda\rt).
 \]
Let $\Sigma=\cup_i \partial Q_i\cap Q_{2L}$ and $\eps_i=sign(\kappa_i-\kappa)$. Using $\xi(x)=d(x,\Sigma)\sum_i  \eps_i\chi_{Q_i}(x)$ as test function  in the dual formulation of $ W^1_{Q_{2L}}$, we obtain 
\[
 W_{Q_{2L}}^1\lt(\sum_i \kappa_i \chi_{Q_i} \lambda, \kappa \lambda\rt)\ge \sum_i \sum_{X\in \Z^d\cap Q_i} d(X,\Sigma) |\kappa_i-\kappa|\ges L^{d+1}\sum_i|\kappa_i-\kappa|.
\]
Taking expectations we find \eqref{eq:claimlowergrid}.\\

\medskip
\textit{Step 2.}[The upper bound] We turn to the corresponding upper bound,
\begin{equation}\label{eq:claimuppergrid}
 \EE\lt[\frac{1}{|Q_{2L}|} W_{Q_{2L}}^p\lt(\sum_i \kappa_i \chi_{Q_i} \lambda, \kappa \lambda\rt)\rt]\les \frac{1}{L^{\frac{d-2}{2}}}.
\end{equation}
One could argue exactly as for \eqref{eq:claimsubbi} but we provide an alternative proof which uses a one-dimensional argument instead of  \eqref{eq:estimCZ}. Notice that this argument could have also been used to give a different proof of \eqref{eq:claimsubref}. \\

\textit{Step 2.1.}[The one-dimensional estimate] Let $\lambda=\sum_{X\in \Z} \delta_X$ and for  
  $\kappa_1$, $\kappa_2>0$ and  $L\gg1$,  define $\kappa=\frac{\kappa_1 \lambda(0,L/2)+\kappa_2\lambda(L/2,L)}{\lambda(0,L)}$. 
  We claim that if $|\kappa_1-\kappa_2|\ll L^{-1}\min(\kappa_1,\kappa_2)$, then
\begin{equation}\label{eq:estim1d}
 W^p_{(0,L)}\lt(\kappa_1 \chi_{(0,L/2)}\lambda +\kappa_2 \chi_{(L/2,L)}\lambda,\kappa \lambda\rt)\les L^2|\kappa_1-\kappa_2|.
\end{equation}
Let us assume without loss of generality that $\kappa_1\ge\kappa_2$. The optimal transport map is essentially symmetric around $L/2$ and is given in $(0,L/2)$ by sending a mass
$(k+1)\frac{\lambda(L/2,L)}{\lambda(0,L)}(\kappa_1-\kappa_2)$  from position $k$ to $k+1$ (which is admissible since by hypothesis $k|\kappa_1-\kappa_2|\ll \min(\kappa_1,\kappa_2)$) so that 
 \[
   W^p_{(0,L)}(\mu,\lambda)\les \sum_{k=0}^{L/2} k |\kappa_1-\kappa_2|\sim L^2|\kappa_1-\kappa_2|.
 \]
This proves \eqref{eq:estim1d}.\\

\textit{Step 2.2.}[Proof of \eqref{eq:claimuppergrid}]
The proof is made recursively by layers. In the first step, we pass from $2^d$ cubes to $2^{d-1}$ rectangles of the form $x+(0,L)\times(0,L/2)^{d-1}$.
For this we remark that for every cube $Q_i$ there is exactly one cube $Q_j$ such that $Q_j=Q_i \pm \frac{L}{2} e_1$ (where $e_i$ is the canonical basis of $\R^d$). Let us focus on $Q_1=(0,L/2)^d$ and $Q_2= \frac{L}{2} e_1+Q_1$. 
Let $R=Q_1\cup Q_2= (0,L)\times(0,L/2)^{d-1}$. Define $\hat{\kappa}= \frac{\kappa_1\lambda(Q_1)+\kappa_2\lambda(Q_2)}{\lambda(R)}=\frac{\mu_1+\mu_2}{\lambda(R)}$. We claim that 
 \[
  \EE\lt[\frac{1}{|R|}W_{R}^2\lt(\kappa_1 \chi_{Q_1}\lambda + \kappa_2 \chi_{Q_2}\lambda, \hat{\kappa} \lambda\rt)\rt]\les \frac{1}{L^{\frac{d-2}{2}}}.
 \]
For this we notice that in $R$,  the measures $\kappa_1 \chi_{Q_1}\lambda + \kappa_2 \chi_{Q_2}\lambda$ and $\hat{\kappa} \lambda$ are constant in the directions orthogonal
to $e_1$ and thus 
\[
 \frac{1}{|R|}W_{R}^p\lt(\kappa_1 \chi_{Q_1}\lambda + \kappa_2 \chi_{Q_2}\lambda, \hat{\kappa} \lambda\rt)\les \frac{1}{L}W_{(0,L)}^p\lt( \kappa_1 \chi_{(0,L/2)}\lambda +  \kappa_2 \chi_{(L/2,L)}\lambda, \hat{\kappa} \lambda\rt).
\]
Since $\kappa_i=\frac{\mu_i}{\lambda(Q_i)}$, by the Cram\'er-Chernoff bounds \eqref{eq:Cramer}, we have $|\kappa_i-1|=O(L^{-d/2})$ with overwhelming probability.  Hence, if $d\ge 3$ we may apply \eqref{eq:estim1d} to get 
\[
  \EE\lt[\frac{1}{|R|}W_{R}^p\lt(\kappa_1 \chi_{Q_1}\lambda + \kappa_2 \chi_{Q_2}\lambda, \hat{\kappa} \lambda\rt)\rt]\les L \EE[|\kappa_1-\kappa_2|]\les \frac{1}{L^{\frac{d-2}{2}}}
\]
which proves the claim.\\
We finally iterate this  this argument $2^d$ times. At every step $k$ we have $2^{d-k}$ rectangles of the form $x+(0,L)^{k}\times(0,L/2)^{d-k}$ and each iteration has a cost of the same order (namely $L^{-(d-2)/2}$).
 Using triangle inequality this concludes the proof of \eqref{eq:claimuppergrid}.

\end{proof}
\begin{remark}
 Notice that  in the proof of \eqref{eq:estim1d}, we  locally move a mass of order $L^{-(d-2)/2}$ which corresponds to the optimal choice of $\theta$ in the proof of \eqref{eq:claimsubbi}. 
\end{remark}

\section{De-Poissonization}\label{Poitostan}
In this section we discuss  how to transfer limit results from matching problems  with Poisson point process, i.e., with a random number of points, to those of with a deterministic number of points. 

We use the following general result.

\begin{proposition}\label{prop:depois}
Let $f: (0, \infty) \times \N^k \to [0, \infty)$, $(L, n) \mapsto f(L|n)$, satisfy the following assumptions:
\begin{enumerate}
\item \emph{($p$-homogeneity)} $f(L|n) = L^p f(1|n)$, for every $n \in \N^k$, $L >0$,
\item \emph{(boundedness)} $f(1 |n ) \lesssim 1$, for every $n \in \N^k$,
\item  \emph{(monotonicity)}   $f(1|n) \le f(1|m)$, for every $m, n \in \N^k$ such that $m_i \le n_i$ for $i =1, \ldots, k$.
\end{enumerate}
Define
\[ f(L) = \EE\sqa{ f(L | N_L)}\]
where $N_L = (N_{L,i})_{i=1}^k$ are i.i.d.\ Poisson random variables with parameter $L^d$. Then,

\[ \liminf_{n \to \infty} n^{\frac{p}{d}} f(1|(n, \ldots, n)) = \liminf_{L \to \infty} f(L), \quad \text{and} \quad \limsup_{n \to \infty} n^{\frac{p}{d}} f(1|(n, \ldots, n)) = \limsup_{L \to \infty} f(L). \]
\end{proposition}

\begin{proof}

Let $0< \delta < L^d$ and introduce the event
\[ A = \cur{  |N_{L,i} - L^d| < \delta \quad \text{for $i =1, \ldots, k$}}.\]
By independence and Poisson tail bounds \eqref{eq:Cramer},
\begin{equation}\label{eq:PA}
\PP(A) \ge \bra{ 1- 2 \exp\lt(- \frac 1 2 \frac{\delta^2}{L^d + \delta}\rt)}^k.
\end{equation}
We decompose
\[ f(L) = \EE\sqa{  f(L | N_L) | A} \PP(A) + \EE\sqa{  f(L | N_L) | A^c } \PP( A^c).\]
If $A$ holds, we use monotonicity of $f(L|n)$ to argue that
\[ L^p f(1| a ) \le  \EE\sqa{  f(L | N_L) | A} \le L^p f(1| b ),\]
where $a=L^d + \delta$, $b = L^d-\delta$. Otherwise, we use $(1)$ and $(2)$ to obtain
\[  0 \le \EE\sqa{  f(L | N_L) | A^c }  \le L^p.\]
Combining these inequalities, we have  
\begin{equation}\label{eq:upper-lower-a-b}  L^p f(1| a ) \PP(A) \le f(L) \le L^p\bra{ f(1| b ) + (1-\PP(A))}.\end{equation}

For any $n \ge 1$, let $L=L(n)$ be such such that $L^d + L^{d/2} \sqrt{ 2 \beta \log L}= n$, for some fixed $\beta > p$. Then, we have from \eqref{eq:upper-lower-a-b} with 
$\delta =L^{d/2}\sqrt{ 2\beta \log L}$,
\begin{equation*}\label{eq:upper-a} f(1| n ) \le \frac{  f(L) }{L^p \PP(A)}.\end{equation*}
As $n \to \infty$, we have
\begin{equation*}\label{eq:quantitative-n-L} 
\frac{ n^{\frac{p}{d}}}{L^{p}} = (1+ L^{-\frac{d}{2}} \sqrt{ 2 \beta \log L})^{\frac{p}{d}} = 1 + O\bra{ \sqrt{ \frac{\log n}{n}}}.\end{equation*}
Moreover, by \eqref{eq:PA},
\begin{equation*} \PP(A) \ge \bra{ 1-2\exp\lt(- \frac 1 2 \frac{\delta^2}{L^d + \delta}\rt)}^k = 1- O(L^{-\beta}) = 1 - O\lt(n^{-\frac{\beta}{d}}\rt).\label{eq:quantitative-p-a}
\end{equation*}
It follows that 
\[\limsup_{n\to \infty}  n^{\frac{p}{d}} f(1|(n, \ldots, n)) \le \limsup_{L \to \infty}  f(L), \quad \text{and} \quad  \liminf_{n\to \infty}  n^{\frac{p}{d}} f(1|(n, \ldots, n)) \le \liminf_{L \to \infty} f(L) .\]
To obtain the converse inequalities, we argue analogously choosing instead $L = L(n)$ such that $L^d - L^{d/2} \sqrt{ 2\beta \log L} = n$.
\end{proof}

We now apply Proposition~\ref{prop:depois} to  matching problems. Let us first consider the case of matching to the reference measure.

\begin{theorem}\label{theo:match-ref-measure-unit-cube}
Let $d \ge 3$, $p\ge 1$, and $(X_{i})_{i \ge 1}$ be i.i.d.\ uniform random variables on $[0,1]^d$. 
Then,  
\[ \lim_{n \to \infty} n^{\frac{p}{d}} \EE\sqa{  W_{[0,1]^d}^p\bra{ \frac 1 n \sum_{i=1}^n \delta_{X_i}, 1}} = \frefinf,\]
with $\frefinf$ as in Theorem~\ref{theo:PoiLeb}.
\end{theorem}

\begin{proof}
Recalling the notation $Q_1= [0,1]^d$, we introduce the function
\begin{align*}
 f(L|n) = L^{p} \EE\sqa{  W_{Q_1}^p\bra{ \frac 1 n \sum_{i=1}^n \delta_{X_i}, 1}}, \quad \text{if $n \ge 1$,}
\end{align*}
and $f(L|0) = 0$. It is clearly bounded and $p$-homogeneous in the sense of Proposition~\ref{prop:depois}. To show monotonicity, let $1\le m\le n$ and use the identity
\[  \frac 1 n \sum_{i=1}^n \delta_{X_i} =  {n \choose m}^{-1} \sum_{\substack{ I \subseteq \cur{1, \ldots, n} \\ |I| = m}}  \frac 1 m \sum_{i \in I} \delta_{X_i}\]
in combination with the convexity of the transportation cost, to obtain
\begin{align*} W_{Q_1}^p\bra{ \sum_{i=1}^n \delta_{X_i}, 1 } & = W_{Q_1}^p\bra{{n \choose m}^{-1} \sum_{\substack{ I \subseteq \cur{1, \ldots, n} \\ |I| = m}} \frac 1 m \sum_{i \in I} \delta_{X_i}, {n \choose m}^{-1} \sum_{\substack{ I \subseteq \cur{1, \ldots, n} \\ |I| = m}} 1 } \\
&  \le {n \choose m}^{-1} \sum_{\substack{ I \subseteq \cur{1, \ldots, n} \\ |I| = m}}  W_{Q_1}^p\bra{ \frac 1 m \sum_{i \in I} \delta_{X_i}, 1 }. 
\end{align*}
Taking expectation yields $f(1|n) \le f(1|m)$, since for  $I \subseteq \cur{1,\ldots, n}$ with $|I| = m$, $\cur{X_i}_{i\in I}$ have the same law as $\cur{X_i}_{1 =1}^m$. 

Let $\mu$ be a Poisson point process of intensity one on $\R^d$ and for $L>1$ let $N_L=\mu(Q_L)$ be a Poisson random variable of parameter $L^d$. For $n \ge 1$, we notice that
\begin{equation}\label{eq:identity-f-fref} \begin{split} f(L|n) & = \EE\sqa{ \left.  W_{Q_L}^p\bra{ \frac{\mu}{\mu(Q_L)}, \frac 1 {|Q_L|}} \right| N_L = n } \\
& = \frac{L^d}{n} \fref(L|n),\end{split}\end{equation}
where we write, extending the notation from Section~\ref{sec:matching-ref},
\[ \fref(L|n) = \EE\sqa{ \left.  \frac{1}{|Q_L|} W_{Q_L}^p\bra{ \mu, \kappa} \right| N_L = n },
\]
with $\kappa = \mu(Q_L)/|Q_L|$ (and $\fref(L|0)=0$). Notice that 
\[
 \fref(L)=\EE[\fref(L|N_L)].
\]
In order to apply Proposition~\ref{prop:depois} and conclude the proof, we argue that
\begin{equation}\label{eq:thesis-depois-ref-limit} \lim_{L \to \infty}  \EE \sqa{\abs{ f\bra{ L | N_L} - \fref\bra{L| N_L}} }   =0.
\end{equation}
To this aim, we first bound $\EE\sqa{ f(L|N_L)}$ uniformly from above as $L \to \infty$. For this we combine the following two simple facts. 
First,  since $f(L|n)\les L^p$, letting $A_L=\{N_L\ge |Q_L|/2\}$, we have 
\[
 \EE\lt[f(L|N_L)|A_L^c\rt]\les L^p.
\]
Second, since by \eqref{eq:Cramer} $\PP[A_L]\ges 1$,
\[
 \EE[f(L|N_L) | A_L]\stackrel{\eqref{eq:identity-f-fref}}{=}\EE\lt[\lt.\frac{L^d}{N_L}\fref(L|N_L) \rt| A_L\rt]
 \les \EE[\fref(L|N_L)]=\fref(L)\les 1,
\]
where in the last inequality we used that $\fref$ is bounded as a consequence of Theorem \ref{theo:PoiLeb}. Therefore,
\[
\begin{split}
\EE[f(L|N_L)]
  & =\EE[ f(L|N_L) | A_L^c] \PP[A_L^c]+ \EE[ f(L|N_L) | A_L] \PP[A_L]\\
  &\les L^p \PP[A_L^c] +1\\
  &\stackrel{\eqref{eq:Cramer}}{\les} 
  L^p\exp(-cL^d)+1\les 1.
\end{split}
\]
Using H\"older inequality with $\frac{1}{q}+\frac{1}{q'}=1$, and the fact that $f(L|N_L)\les L^p$,
\[
\begin{split}
 \EE\sqa{\lt|f(L|N_L)-\fref(L|N_L)\rt|}& \stackrel{\eqref{eq:identity-f-fref}}{=}\EE\sqa{\lt|\frac{N_L}{|Q_L|}-1\rt|f(L|N_L)}\\
 &\le \EE\lt[\lt|\frac{N_L}{|Q_L|}-1\rt|^{q'}\rt]^\frac{1}{q'}\EE\sqa{ f\bra{ L |N_L}^{q} }^{\frac{1}{q}}.\\
 & \stackrel{\eqref{eq:momentPoi}}{\les_q} L^{-\frac{d}{2}} L^{\frac{p(q-1)}{q}}\EE\sqa{ f\bra{ L | N_L} }^{\frac{1}{q}}\les_q L^{- \frac{d}2 +\frac{p(q-1)}{q} }.
\end{split}\]
Choosing $q$ close enough to $1$, in particular $1<q<\frac{2p}{2p-d}$ if $p>\frac{d}{2}$, we get \eqref{eq:thesis-depois-ref-limit}. 
 \end{proof}

Arguing similarly, we obtain the corresponding result for the bi-partite matching on the unit cube, that is Theorem \ref{theo:biintro}, which we restate for the reader's convenience.

\begin{theorem}\label{theo:match-bipartite-unit-cube}
Let $d \ge 3$, $p\ge 1$, and $(X_{i})_{i \ge 1}$, $(Y_{i})_{i \ge 1}$, be independent uniform random variables on $[0,1]^d$. 
Then,  
\[ \lim_{n \to \infty} n^{\frac{p}{d}} \EE\sqa{  W_{[0,1]^d}^p\bra{ \frac 1 n \sum_{i=1}^n \delta_{X_i}, \frac 1 n \sum_{i=1}^n \delta_{Y_i}}}  = \fbiinf,\]
with $\fbiinf$ as in Theorem~\ref{theo:bi}.
\end{theorem}

\begin{proof}
The proof is very similar to that of Theorem~\ref{theo:match-ref-measure-unit-cube}, but in this case we define the function 
\[  f(L| n_1, n_2) = L^p \EE\sqa{  W_{[0,1]^d}^p\bra{ \frac 1 {n_1} \sum_{i=1}^{n_1} \delta_{X_i}, \frac 1 {n_2} \sum_{i=1}^{n_2} \delta_{Y_i}}}\]
for $n_1$, $n_2 \ge 1$, and let $f(L | n_1, n_2) = 0$ otherwise. It is clearly bounded and $p$-homogeneous. 
Using the identities
\[ \frac 1 {n_1} \sum_{i=1}^{n_1} \delta_{X_i} =  {n_1 \choose m_1}^{-1}{n_2 \choose m_2}^{-1}  \sum_{\substack{ I_1 \subseteq \cur{1, \ldots, n_1} \\ |I| = m_1}} \sum_{\substack{ I_2 \subseteq \cur{1, \ldots, n_2} \\ |I| = m_2}}  \frac 1 {m_1} \sum_{i \in I_1} \delta_{X_i}\]
\[ \frac 1 {n_2} \sum_{i=1}^{n_2} \delta_{Y_i} =  {n_1 \choose m_1}^{-1}{n_2 \choose m_2}^{-1}  \sum_{\substack{ I_1 \subseteq \cur{1, \ldots, n_1} \\ |I| = m_1}} \sum_{\substack{ I_2 \subseteq \cur{1, \ldots, n_2} \\ |I| = m_2}}  \frac 1 {m_2} \sum_{i \in I_2} \delta_{Y_i}\]
we obtain monotonicity arguing analogously as in Theorem~\ref{theo:match-ref-measure-unit-cube}. The proof is then concluded as in  Theorem \ref{theo:match-ref-measure-unit-cube} and we omit the details.
\end{proof}
%

\begin{remark}[Convergence rates]\label{rem:rates}
An inspection of the proof of Proposition~\ref{prop:depois} shows that one can transfer rates of convergence (even only one-sided)
from the Poisson case to that of a fixed number of points. In the case of matching to the reference measure this leads to the inequality
\[  \frefinf  \le n^{\frac{p}{d}}\EE\sqa{  W_{Q_1}^p\bra{ \frac 1 n \sum_{i=1}^n \delta_{X_i}, 1}} + C n^{ \frac{2-d}{2d}},\]
while for  the bi-partite matching we obtain
\[  \fbiinf  \le   n^{\frac{p}{d}} \EE\sqa{  W_{Q_1}^p\bra{ \frac 1 n \sum_{i=1}^n \delta_{X_i}, \frac 1 n \sum_{i=1}^n \delta_{Y_i}}}+  C n^{\frac{ 2-d}{2dp}},\]
for some constant $C \ge 0$. Besides being one-sided bounds, these are still far from the conjectured rates in \cite{CaLuPaSi14}, which for $p=2$ read
\[ n^{\frac{2}{d}} \EE\sqa{  W_{Q_1}^2\bra{ \frac 1 n \sum_{i=1}^n \delta_{X_i}, \frac 1 n \sum_{i=1}^n \delta_{Y_i}}}  = \fbiinf  +  C n^{\frac{2-d}{d}} + o(n^{\frac{2-d}{d}}),\]
with an explicit constant $C$.
\end{remark}

\begin{remark}[Strong convergence]\label{rem:strong-convergence}
If $p<d$, standard concentration of measure arguments allow to obtain strong convergence from convergence of the expected values (see also \cite{AmStTr16} for a similar argument in the case $p=d=2$). Let us consider for example the case of bi-partite matching, and show that, both  $\PP$-a.s.\ and in $L^1(\PP)$,
\begin{equation}\label{eq:strong} \lim_{n \to \infty}  n^{\frac{p}{d}} W_{Q_1}^p\bra{ \frac 1 n \sum_{i=1}^n \delta_{X_i}, \frac 1 n \sum_{i=1}^n \delta_{Y_i}} = \fbiinf,\end{equation}
with $(X_i)_{i \ge 1}$, $(Y_i)_{i\ge 1}$ independent uniform random variables on $Q_1 = [0,1]^d$.

For any $n \ge 1$, the function
\begin{equation*}\label{eq:function-concentration}
[0,1]^{d \times 2n} \ni (x_1, \ldots, x_n, y_1, \ldots, y_n) \mapsto n^{\frac{1}{d}} W_{p}\bra{ \mu_{\bra{x_i}_{i=1}^n},  \mu_{\bra{y_i }_{i=1}^n}}, 
\end{equation*}

is $2 n^{\frac{1}{d}-\min\cur{\frac 1 2, \frac 1 p}}$-Lipschitz with respect to the Euclidean distance, where we write
\[ \mu_{\bra{x_i}_{i=1}^n} = \frac 1 n \sum_{i=1}^n \delta_{x_i},\quad \text{for $x_i \in [0,1]^d$.}\]
This relies on the triangle inequality \eqref{eq:triangle} and  the fact that 
 \[ (x_i)_{i=1}^n \mapsto \mu_{\bra{x_i}_{i=1}^n} \quad \text{is   $n^{-\min\cur{\frac 1 2, \frac 1 p}}$-Lipschitz}\]
  if we endow the set of probability measures on $[0,1]^d$  with the Wasserstein distance of order $p$. Indeed, for $(x_i)_{i=1}^n, (y_i)_{i=1}^n \in [0,1]^{d \times n}$, then
\[  W_{p}\bra{\mu_{\bra{x_i}_{i=1}^n}, \mu_{\bra{y_i}_{i=1}^n }} \le \bra{ \frac 1 n \sum_{i=1}^n  |x_i-y_i|^p}^{\frac{1}{p}}
\le n^{-\min\cur{\frac 1 2, \frac 1 p}} \bra{ \sum_{i=1}^n  |x_i-y_i|^2}^{\frac{1}{2}}.\]

Gaussian concentration for the uniform measure on the unit cube \cite[Prop. 2.8]{ ledoux2001concentration} yields that if 
\[ Z_n = n^{\frac{1}{d}}W_p\bra{\mu_{\bra{X_i}_{i=1}^n}, \mu_{\bra{Y_i}_{i=1}^n} },\]
then, for $r >0$,
\[ \PP \bra{ \abs{Z_n - \EE\sqa{Z_n}} \ge  r } \le 2 \exp\lt(- c n^{2\lt(\min\cur{\frac 1 2, \frac 1 p}-\frac{1}{d}\rt)}r^2\rt),\]
where $c>0$ is an absolute constant. A standard application of Borel-Cantelli Lemma gives that, if  $1 \le p <d$ and $d \ge 3$, then 
\begin{equation}\label{eq:as} \lim_{n \to \infty} \bra{Z_n -\EE\sqa{Z_n}} = 0, \quad \text{$\PP$-a.s.}\end{equation}
Moreover, using the layer-cake formula, we obtain the inequality
\begin{equation}\label{eq:poincare} \EE\sqa{ \abs{ Z_n - \EE\sqa{Z_n}}^p } \les  n^{\frac{p}{d} -\min\cur{\frac{p}{2}, 1} },\end{equation}
which is infinitesimal as $n \to \infty$.  To conclude, it is sufficient to argue that $\lim_{n \to \infty} \EE\sqa{Z_n} = (\fbiinf)^{1/p}$, since it yields by \eqref{eq:poincare} and \eqref{eq:as} that $\lim_{n \to \infty} Z_n =  (\fbiinf)^{1/p}$ in $L^p(\PP)$ and $\PP$-a.s.\ and hence \eqref{eq:strong}.
By Theorem~\ref{theo:match-ref-measure-unit-cube},  $\lim_{n \to \infty} \EE\sqa{Z_n^p} = \fbiinf$. By Jensen's inequality,
\[ \limsup_{n \to \infty} \EE\sqa{Z_n}^p \le \lim_{n \to \infty} \EE\sqa{Z_n^p} =  \fbiinf.\]
By the elementary inequality \eqref{eq:elementary}, we have, for any $\eps>0$, 
\[  Z_n^p \le \bra{ \EE\sqa{Z_n} + |Z_n - \EE\sqa{Z_n}|}^p \le \bra{1+\eps}\EE\sqa{Z_n}^p + \frac{C}{\eps^{p-1}} |Z_n - \EE\sqa{Z_n}|^p.\]
Taking expectation and letting $n \to \infty$, using \eqref{eq:poincare}, we obtain
\[ \fbiinf = \lim_{n \to \infty} \EE\sqa{Z_n^p} \le \bra{1+\eps} \liminf_{n\to \infty} \EE\sqa{Z_n}^p.\]
Letting $\eps\to 0$ we conclude.
%
%


\end{remark}

\section*{Acknowledgments}
 M.G was partially supported by the  project 
ANR-18-CE40-0013 SHAPO financed by the French Agence Nationale de la Recherche (ANR) and by the LYSM LIA AMU CNRS ECM INdAM. 
M.G. also thanks the Centro de Giorgi in Pisa for its hospitality.  D.T.\ was partially supported by the University of Pisa, Project PRA 2018-49, and Gnampa project 2019 ``Propriet\`a analitiche e geometriche di campi aleatori''.
\bibliographystyle{amsplain}

\bibliography{OT}
 \end{document}